\documentclass{amsart}
\usepackage{amsmath}
\usepackage{amsfonts}

\newtheorem{theorem}{Theorem}[section]
\newtheorem{corollary}[theorem]{Corollary}
\newtheorem{definition}[theorem]{Definition}
\newtheorem{lemma}[theorem]{Lemma}

\newtheorem{remark}[theorem]{Remark}

\input{tcilatex}

\begin{document}
\title[\textbf{On the unbounded of a class of Fourier integral operator on }$%
L^{2}\left( \mathbb{R}^{n}\right) $]{\textbf{On the unbounded of a class of
Fourier integral operator on }$L^{2}\left( \mathbb{R}^{n}\right) $ \textbf{\ 
}}
\author{\textbf{A. Senoussaoui }}
\address{Universit\'{e} d'Oran, Facult\'{e} des Sciences\\
D\'{e}partement de\ Math\'{e}matiques. B.P. 1524 El-Mnaouer, Oran, ALGERIA.}
\email{senoussaoui\_abdou@yahoo.fr, senoussaoui.abderahmane@univ-oran.dz}
\thanks{The paper was realized during the stay of the author at Universit%
\'{e} Libre de Bruxelles. We would like to thank Professor J.-P. Gossez for
severals discussions. }

\begin{abstract}
In this paper, we give an example of Fourier integral operator with a symbol
belongs to $\bigcap\limits_{0<\rho <1}S_{\rho ,1}^{0}$ that cannot be
extended as a bounded operator on $L^{2}\left( \mathbb{R}^{n}\right) .$
\end{abstract}

\maketitle

\section{Introduction}

A Fourier integral operator is a singular integral operator of the form%
\begin{equation*}
I\left( a,\phi \right) u\left( x\right) =\dint \dint e^{i\phi \left(
x,y,\theta \right) }a\left( x,y,\theta \right) u\left( y\right) dyd\theta
\end{equation*}%
defined under certain assumptions on the regularity and asymptotic
properties of the phase function $\phi $ and the amplitude function $a.$
Here $\theta $ plays the role of the covariable.

Fourier integral operators are more general than pseudodifferantial
operators, where the phase function is of the form $\left\langle x-y,\theta
\right\rangle .$

Let us denote by $S_{\rho ,\delta }^{m}\left( \mathbb{R}^{n_{1}}\times 
\mathbb{R}^{n_{2}}\times \mathbb{R}^{N}\right) $ the space of \newline
$a\left( x,y,\theta \right) \in C^{\infty }\left( \mathbb{R}^{n_{1}}\times 
\mathbb{R}^{n_{2}}\times \mathbb{R}^{N}\right) ,$ satisfying%
\begin{equation*}
\left\vert \partial _{x}^{\alpha }\partial _{y}^{\beta }\partial _{\theta
}^{\gamma }a(x,y,\theta )\right\vert \leq C_{\alpha ,\beta ,\gamma }\lambda
^{m-\rho \left\vert \gamma \right\vert +\delta (\left\vert \alpha
\right\vert +\left\vert \beta \right\vert )}(\theta ),\text{ }\forall \left(
\alpha ,\beta ,\gamma \right) \in \mathbb{N}^{n_{1}}\times \mathbb{N}%
^{n_{2}}\times \mathbb{N}^{N},
\end{equation*}%
where $\lambda \left( \theta \right) =\left( 1+\left\vert \theta \right\vert
\right) .$

The phase function $\phi \left( x,y,\theta \right) $ is assumed to be a-$%
C^{\infty }\left( \mathbb{R}^{n_{1}}\times \mathbb{R}^{n_{2}}\times \mathbb{R%
}^{N},\mathbb{R}\right) $ real function, homogeneous in $\theta $ of degree $%
1.$

Since 1970, many efforts have been made by several authors in order to study
this type of operators (see, e.g., \cite{AsFu,Du,Ha,He,Ho1}).

For the Fourier integral operators, an interesting question is under which
conditions on $a$ and $\phi $ these operators are bounded on $L^{2}$ or on
the Sobolev spaces $H^{s}$.

It was proved in \cite{Ho3} that all pseudodifferential operators with
symbol in $S_{\rho ,\delta }^{0}$ are bounded on $L^{2}$ if $\delta <\rho .$
When $0<\delta =\rho <1,$ Calderon and Vaillancourt \cite{CaVa} \ have
proved that all pseudodifferantial operators with symbol in $S_{\rho ,\rho
}^{0}$ are bounded on $L^{2}.$ On the other hand, Kumano-Go \cite{Ko} has
given a pseudodifferential operator with symbol belonging to $%
\bigcap\limits_{0<\rho <1}S_{\rho ,1}^{0}$ which is not bounded on $%
L^{2}\left( \mathbb{R}\right) $.

For Fourier integral operators, it has been proved in \cite{AsFu} that the
operator \newline
$I\left( a,\phi \right) $ $:L^{2}\longrightarrow L^{2}$ is bounded if $%
\delta =m=\rho =0.$ Recently, M. Hasanov \cite{Ha} constructed a class of
unbounded Fourier integral operators on $L^{2}\left( \mathbb{R}\right) $
with an amplitude in $S_{1,1}^{0}.$

For $u\in C_{0}^{\infty }\left( \mathbb{R}^{n}\right) $, the integral
operators%
\begin{equation}
I\left( a,S\right) \varphi \left( x\right) =\int e^{iS\left( x,\theta
\right) }a\left( x,y,\theta \right) \mathcal{F}\varphi \left( \theta \right)
d\theta  \label{0.1}
\end{equation}%
appear naturally in the expression of the solutions of hyperbolic partial
differential equations (see \cite{EgSh1,EgSh2,MeSe}).

If we write formally the expression of the Fourier transformation $\mathcal{F%
}u\left( \theta \right) $ in $\left( \ref{0.1}\right) ,$ we obtain the
following Fourier integral operators%
\begin{equation}
I\left( a,S\right) u\left( x\right) =\iint e^{i\left( S\left( x,\theta
\right) -y\theta \right) }a\left( x,y,\theta \right) u\left( y\right)
dyd\theta  \label{0.2}
\end{equation}%
in which the phase function has the form $\phi \left( x,y,\theta \right)
=S\left( x,\theta \right) -y\theta $. We note that in \cite{MeSe2}, we have
studied the $L^{2}$-boundedness and $L^{2}$-compactness of a class of
Fourier integral operator of the form $\left( \ref{0.2}\right) .$

In this article we give an example of a Fourier integral operator, in higher
dimension, of the form $\left( \ref{0.1}\right) $ with symbol $a\left(
x,\theta \right) \in $ $\bigcap\limits_{0<\rho <1}S_{\rho ,1}^{0}$
independent on $y,$ that cannot be extended to a bounded operator in $%
L^{2}\left( \mathbb{R}^{n}\right) ,n\geq 1.$ Here we take the phase function
in the form of separate variable $S\left( x,\theta \right) =\varphi \left(
x\right) \psi \left( \theta \right) .$

\section{\textbf{The boundedness on }$C_{0}^{\infty }(\mathbb{R}^{n})$%
\textbf{\ and on }$D^{\prime }(\mathbb{R}^{n})$}

If $\varphi \in C_{0}^{\infty }\left( \mathbb{R}^{n}\right) ,$ we consider
the following integral transformations%
\begin{eqnarray}
\left( I\left( a,S\right) \varphi \right) \left( x\right) &=&\int_{\mathbb{R}%
^{N}}e^{iS\left( x,\theta \right) }a\left( x,\theta \right) \mathcal{F}%
\varphi \left( \theta \right) d\theta  \notag \\
&=&\int_{\mathbb{R}^{N}\times \mathbb{R}^{n}}e^{i\left( S\left( x,\theta
\right) -y\theta \right) }a\left( x,\theta \right) \varphi \left( y\right)
d\theta dy  \label{1.1}
\end{eqnarray}%
for $x\in \mathbb{R}^{n}$and $N\in \mathbb{N}$.

In general the integral $(\ref{1.1})$ is not absolutely convergent, so we
use the technique of the oscillatory integral developed by L.H\"{o}rmander
in \cite{Ho1}. The phase function $S$ and the amplitude $a$ are assumed to
satisfy the hypothesis

$(H1)$ $\ \ \ S\in C^{\infty }\left( \mathbb{R}_{x}^{n}\times \mathbb{R}%
_{\theta }^{N},\mathbb{R}\right) $ ($S$ real function)

$(H2)\;\;$ $\forall \beta \in \mathbb{N}^{N},$ $\exists C_{\beta }>0;$ 
\begin{equation*}
\left\vert \partial _{\theta }^{\beta }S\left( x,\theta \right) \right\vert
\leq C_{\beta }\left( x\right) \lambda ^{\left( 1-\left\vert \beta
\right\vert \right) _{+}}\left( \theta \right) \text{ , }\forall \left(
x,\theta \right) \in \mathbb{R}_{x}^{n}\times \mathbb{R}_{\theta }^{N}
\end{equation*}%
where $\lambda \left( \theta \right) =\left( 1+\left\vert \theta \right\vert
\right) $ and $\left( 1-\left\vert \beta \right\vert \right) _{+}=\max
\left( 1-\left\vert \beta \right\vert ,0\right) .$

$(H3)\;\;$ $S$ satisfies 
\begin{equation*}
\left( \frac{\partial S}{\partial x},\frac{\partial S}{\partial \theta }%
-y\right) \neq 0,\text{ }\forall \left( x,\theta \right) \in \mathbb{R}%
_{x}^{n}\times \left( \mathbb{R}_{\theta }^{N}\backslash \left\{ 0\right\}
\right) .
\end{equation*}

\begin{remark}
If the phase function $S\left( x,\theta \right) $ is homogeneous in $\theta $
of degree 1, then it satisfies $(H2).$
\end{remark}

For any open $\Omega $ of $\mathbb{R}_{x}^{n}\times \mathbb{R}_{\theta
}^{N}, $ $m\in \mathbb{R}$, $\rho >0$ and $\delta \geq 0$ we set

\begin{equation*}
S_{\rho ,\delta }^{m}\left( \Omega \right) =\left\{ 
\begin{array}{ccc}
a\in C^{\infty }\left( \Omega \right) ;\text{ } & \forall (\alpha ,\beta
)\in \mathbb{N}^{n}\times \mathbb{N}^{N}, & \exists C_{\alpha ,\beta }>0; \\ 
\text{ }\left\vert \partial _{x}^{\alpha }\partial _{\theta }^{\beta
}a(x,\theta )\right\vert \leq & C_{\alpha ,\beta }\lambda ^{m-\rho
\left\vert \beta \right\vert +\delta \left\vert \alpha \right\vert }(\theta
). & 
\end{array}%
\right\}
\end{equation*}

\begin{theorem}
\textit{If }$S$ \textit{satisfies }$(H1),(H2),(H3)$\textit{\ and if }$a\in
S_{\rho ,\delta }^{m}\left( \mathbb{R}_{x}^{n}\times \mathbb{R}_{\theta
}^{N}\right) ,$\textit{\ then }$I\left( a,\phi \right) $\textit{\ }is a
continuous operator from $C_{0}^{\infty }\left( \mathbb{R}^{n}\right) $ to $%
C^{\infty }\left( \mathbb{R}^{n}\right) $ and from $\mathcal{E}^{\prime
}\left( \mathbb{R}^{n}\right) $ to $D^{\prime }\left( \mathbb{R}^{n}\right)
, $ where $\rho >0$ and $\delta <1.$
\end{theorem}

\begin{proof}
See \cite{Ha}, \cite[pages 50-51]{EgSh1}.
\end{proof}

\begin{corollary}
\textit{Let }$\varphi \left( x\right) ,\psi \left( \theta \right) \in
C^{\infty }\left( \mathbb{R}^{n},\mathbb{R}\right) $\textit{\ two functions, 
}$\psi \left( \theta \right) $ \textit{is homogeneous of degree }$1$ $(\psi
\left( \theta \right) \neq 0)$\textit{\ and }$\varphi \left( x\right) $
satisfies%
\begin{equation}
\varphi ^{\prime }\left( x\right) \neq 0,\text{ }\forall x\in \mathbb{R}^{n}.
\label{1.2}
\end{equation}%
\textit{Then the operator}%
\begin{equation}
\left( Fu\right) \left( x\right) =\int_{\mathbb{R}^{n}}e^{i\varphi \left(
x\right) \psi \left( \theta \right) }a\left( x,\theta \right) \mathcal{F}%
u\left( \theta \right) d\theta ,\text{ \ }u\in \mathcal{S}(\mathbb{R}^{n})
\label{1.3}
\end{equation}%
\textit{maps continuously} $C_{0}^{\infty }\left( \mathbb{R}^{n}\right) $ to 
$C^{\infty }\left( \mathbb{R}^{n}\right) $ \textit{and} $\mathcal{E}^{\prime
}\left( \mathbb{R}^{n}\right) $ to $D^{\prime }\left( \mathbb{R}^{n}\right) $
\textit{for every} $a\in S_{\rho ,\delta }^{m}\left( \mathbb{R}_{x}\times 
\mathbb{R}_{\theta }\right) ,$ where $\rho >0$ and $\delta <1.$
\end{corollary}

\begin{proof}
For the phase function $S(x,\theta )=\varphi \left( x\right) \psi \left(
\theta \right) ,$ $(H1),$ $(H2)$ and $(H3)$ are satisfied.
\end{proof}

\section{\textbf{The unboundedness of the operator }$F$\textbf{\ on }$%
L^{2}\left( \mathbb{R}^{n}\right) $\textbf{\ }}

In this section we shall construct a symbol $a\left( x,\theta \right) $ in
the H\"{o}rmander space $\bigcap\limits_{0<\rho <1}S_{\rho ,1}^{0}\left( 
\mathbb{R}_{x}^{n}\times \mathbb{R}_{\theta }^{n}\right) ,$ such that the
Fourier integral operator $F$ can not be extended as a bounded operator in $%
L^{2}\left( \mathbb{R}^{n}\right) .$

\setcounter{equation}{0}

\begin{lemma}
\label{LemmaKumano}(\textit{Kumano-Go \cite{Ko}). Let }$f_{0}\left( t\right) 
$\textit{\ be a continuous function on }$\left[ 0,1\right] $\textit{\ such
that}%
\begin{equation}
f_{0}\left( 0\right) =0,\text{ }f_{0}\left( t\right) >0\text{\textit{\ in} }%
\left] 0,1\right] .  \label{2.1}
\end{equation}%
\textit{Then, there exists a continuous function }$b\left( t\right) $\textit{%
\ on }$\left[ 0,1\right] $\textit{\ such that }$b\left( t\right) $\textit{\
satisfies the conditions}%
\begin{equation}
\left\{ 
\begin{array}{l}
f_{0}\left( t\right) \leq b\left( t\right) \text{ \ \textit{on} }\left[ 0,1%
\right] , \\ 
b\in C^{\infty }\left( \left] 0,1\right] \right) ,\text{ }b\left( 0\right)
=0,\text{ }b^{\prime }\left( t\right) >0\text{ \textit{in }}\left] 0,1\right]
, \\ 
\left\vert b^{\left( n\right) }\left( t\right) \right\vert \leq C_{n}t^{-n}%
\text{ \textit{in} }\left] 0,1\right] ,\text{ }n\in \mathbb{N}^{\ast },\text{
}C_{n}>0.%
\end{array}%
\right.  \label{C}
\end{equation}
\end{lemma}

\begin{definition}
It is obvious that an operator $A$ is extended as a bounded operator in $%
L^{2}\left( \mathbb{R}^{n}\right) $ if and only if there exists a constant $%
C>0$ such that%
\begin{equation}
\left\Vert Au\right\Vert _{L^{2}}\leq C\left\Vert u\right\Vert _{L^{2}}\text{
for any }u\in \mathcal{S}(\mathbb{R}^{n}).  \label{2.4}
\end{equation}
\end{definition}

\begin{theorem}
\label{Theorem1}\textit{Let A be an operator, given at least for }$x\in %
\left] 0,\beta \right[ ^{n}$\textit{\ (}$\beta <1$)$,$\textit{\ by}%
\begin{equation*}
\left( Au\right) \left( x\right) =\int_{\mathbb{R}^{n}}u\left( g\left(
x\right) z\right) \mathcal{F}\Psi \left( z\right) dz,\text{ \ }u\in \mathcal{%
S}(\mathbb{R}^{n}),
\end{equation*}%
we denote here by $\left] 0,\beta \right[ ^{n}=\prod\limits_{j=1}^{n}\left]
0,\beta \right[ $.

\textit{If }$\Psi \in \mathcal{S}(\mathbb{R}^{n}),$ $\Psi \left( 0\right) =1$
\textit{and the function }$g$ $\mathit{\in C}^{0}\left( \left] 0,\beta %
\right[ ^{n},\mathbb{R}_{+}\right) $\textit{\ satisfies}%
\begin{equation}
\left\{ 
\begin{array}{l}
\lim\limits_{\left\vert x\right\vert \rightarrow 0^{+}}\frac{g\left(
x\right) }{\left\vert x\right\vert }=0, \\ 
\forall i\in \left\{ 1,...,n\right\} ;\text{ }x_{i}\longrightarrow g\left(
x_{1},..,x_{i},...,x_{n}\right) \text{ \textit{is increasing on }}\left]
0,\beta \right[ \text{.}%
\end{array}%
\right. \text{ }  \label{2.6}
\end{equation}%
\textit{Then the operator }$A$\textit{\ cannot be extended to a bounded
operator in }$L^{2}\left( \mathbb{R}^{n}\right) .$
\end{theorem}

\begin{proof}
Using the Fourier inversion formula in $\mathcal{S}(\mathbb{R}^{n})$, we have%
\begin{equation*}
\left( Au\right) \left( x\right) =\int_{\mathbb{R}^{n}}u\left( g\left(
x\right) z\right) \mathcal{F}\Psi \left( z\right) dz,\text{ \ }u\in \mathcal{%
S}(\mathbb{R}^{n}).
\end{equation*}%
Then there exists a constant $N_{0}>0$ such that%
\begin{equation}
\left\vert \left( 2\pi \right) ^{-\frac{n}{2}}\int_{\left[ -N,N\right] ^{n}}%
\mathcal{F}\Psi \left( z\right) dz\right\vert \geq \beta \text{ \ \ for any }%
N\geq N_{0}\text{.}  \label{2.7}
\end{equation}%
Setting for $\varepsilon >0$ 
\begin{equation*}
u_{\varepsilon }\left( z\right) =\left\{ 
\begin{array}{l}
\left( 2\pi \right) ^{-\frac{n}{2}},\text{ for }z\in \left[ -\varepsilon
,\varepsilon \right] ^{n} \\ 
0,\text{ \ \ for }z\notin \left[ -\varepsilon ,\varepsilon \right] ^{n}%
\end{array}%
\right. .
\end{equation*}%
Then, using the density of $\mathcal{S}(\mathbb{R}^{n})$ in $L^{2}(\mathbb{R}%
^{n}),$ we see that $\left( Au_{\varepsilon }\right) \left( x\right) $ must
be%
\begin{equation}
\left( Au_{\varepsilon }\right) \left( x\right) =\left( 2\pi \right) ^{-%
\frac{n}{2}}\int_{\left[ -\frac{\varepsilon }{g\left( x\right) },\frac{%
\varepsilon }{g\left( x\right) }\right] ^{n}}\mathcal{F}\Psi \left( z\right)
dz\text{ \ for }x\in \left] 0,\beta \right[ ^{n}.  \label{2.8}
\end{equation}%
By $\left( \ref{2.6}\right) $ for any $p\in \mathbb{N}^{\ast }$ there exists
a small $\varepsilon _{p}\geq 0$ such that%
\begin{equation*}
\frac{\varepsilon _{p}}{g\left( p\varepsilon _{p},...,p\varepsilon
_{p}\right) }\geq N_{0}\text{ and }p\varepsilon _{p}\leq \beta .
\end{equation*}%
It follows from the condition $\left( \ref{2.6}\right) $ that%
\begin{equation*}
\frac{\varepsilon _{p}}{g\left( x\right) }\geq \frac{\varepsilon _{p}}{%
g\left( p\varepsilon _{p},...,p\varepsilon _{p}\right) }\geq N_{0},\text{ \
holds for }x\in \left] 0,p\varepsilon _{p}\right] ^{n},
\end{equation*}%
so that, using $\left( \ref{2.7}\right) $ and $\left( \ref{2.8}\right) ,$ we
have%
\begin{equation}
\left\Vert Au_{\varepsilon _{p}}\right\Vert _{L^{2}}^{2}=\int_{\mathbb{R}%
^{n}}\left\vert Au_{\varepsilon _{p}}\left( x\right) \right\vert ^{2}dx\geq
\int_{\left[ 0,p\varepsilon _{p}\right] ^{n}}\left\vert Au_{\varepsilon
_{p}}\left( x\right) \right\vert ^{2}dx\geq \beta ^{2}\left( p\varepsilon
_{p}\right) ^{n}.  \label{2.9}
\end{equation}%
Assume that $A$ is bounded on $L^{2}(\mathbb{R}^{n}).$ According to $\left( %
\ref{2.4}\right) $ there exists $C>0$ such that: 
\begin{equation*}
\beta ^{2}\left( p\varepsilon _{p}\right) ^{n}\leq \left\Vert
Au_{\varepsilon _{p}}\right\Vert _{L^{2}}^{2}\leq C^{2}\left( 2\varepsilon
_{p}\right) ^{n}\text{ for any }p.
\end{equation*}%
Which is a contradiction$.$
\end{proof}

Let $K\left( t\right) $ be a function from $\mathcal{S}(\mathbb{R})$ such
that $K\left( t\right) =1$ on $\left[ -\delta ,\delta \right] $ ($\delta <1$%
), $b\left( t\right) \in C^{0}\left( \left[ 0,1\right] \right) $ be a
continuous function satisfying conditions $\left( \ref{C}\right) $ and $%
\varphi \left( x\right) ,\psi \left( \theta \right) \in C^{\infty }\left( 
\mathbb{R}^{n},\mathbb{R}\right) $ with $\psi \left( \theta \right) $
homogeneous of degree $1$ $(\psi \left( \theta \right) \neq 0).$ We assume
that $\varphi \left( x\right) $ satisfies 
\begin{equation}
\left\vert \varphi \left( x\right) \right\vert \leq C\left\vert x\right\vert 
\text{ for }\left\vert x\right\vert \leq 1.  \label{Es}
\end{equation}%
We remark that if the function $\varphi \left( x\right) $ is homogeneous of
degree $1,$ then it satisfies $\left( \ref{Es}\right) .$

For $x=\left( x_{1},...,x_{n}\right) ,\theta =\left( \theta _{1},...\theta
_{n}\right) \in \mathbb{R}^{n},$ set%
\begin{equation*}
q\left( x,\theta \right) =e^{-i\varphi \left( x\right) \psi \left( \theta
\right) }\prod\limits_{j=1}^{n}K\left( b\left( \left\vert x\right\vert
\right) \left\vert x\right\vert \theta _{j}\right)
\end{equation*}

\begin{lemma}
\label{Lemma1}The function $q\in C^{\infty }\left( \left[ -1,1\right]
^{n}\times \mathbb{R}_{\theta }^{n}\right) $ and the following estimate
holds:%
\begin{gather}
\forall (\alpha ,\beta )\in \mathbb{N}^{n}\times \mathbb{N}^{n},\text{ }%
\exists C_{\alpha ,\beta }>0;  \notag \\
\left\vert \partial _{x}^{\alpha }\partial _{\theta }^{\beta }q(x,\theta
)\right\vert \leq C_{\alpha ,\beta }\lambda ^{\left\vert \alpha \right\vert
-\left\vert \beta \right\vert }\left( \theta \right) \left\{ b\left( \lambda
^{-1}\left( \theta \right) \right) \right\} ^{-\left\vert \beta \right\vert }%
\text{ on }\left[ -1,1\right] ^{n}\times \mathbb{R}_{\theta }^{n}.
\label{2.10}
\end{gather}
\end{lemma}

\begin{proof}
We adopt here the same strategy of Kumano-Go \cite{Ko} lemma 2.

Since $\psi \left( \theta \right) $ is homogeneous of degree $1$ $(\psi
\left( \theta \right) \neq 0),$ it will be sufficient to check the estimate
for $n=1$ on\textit{\ }$\left[ -1,1\right] \times \mathbb{R}_{\theta },$ i.e.%
\begin{gather}
\forall (j,k)\in \mathbb{N}\times \mathbb{N},\text{ }\exists C_{j,k}>0; 
\notag \\
\left\vert \partial _{x}^{j}\partial _{\theta }^{k}\left[ e^{-i\varphi
\left( x\right) \theta }K\left( b\left( \left\vert x\right\vert \right)
x\theta \right) \right] \right\vert \leq C_{j,k}\lambda ^{j-k}\left( \theta
\right) \left\{ b\left( \lambda ^{-1}\left( \theta \right) \right) \right\}
^{-k}.  \label{Es2}
\end{gather}

Since $K\left( t\right) \in \mathcal{S}(\mathbb{R})$ and $K^{\left( n\right)
}\left( t\right) =0$ on $\left[ -\delta ,\delta \right] ,$ $n\in \mathbb{N}%
^{\ast },$ then 
\begin{eqnarray}
\left\vert t^{l}K\left( t\right) \right\vert &\leq &C_{l},\text{ \ }\forall
l\in \mathbb{N}  \label{2.11} \\
\left\vert t^{l}K^{\left( n\right) }\left( t\right) \right\vert &\leq
&C_{l,n},\text{ }\forall n\in \mathbb{N}^{\ast },\forall l\in \mathbb{Z}%
\text{.}  \label{2.12}
\end{eqnarray}%
By Leibnitz's formula we have%
\begin{gather*}
\partial _{x}^{j}\partial _{\theta }^{k}\left[ e^{-i\varphi \left( x\right)
\theta }K\left( b\left( \left\vert x\right\vert \right) x\theta \right) %
\right] = \\
\sum\limits_{\substack{ j_{1}+j_{2}=j  \\ k_{1}+k_{2}=k}}%
C_{j_{1},j_{2},k_{1},k_{2}}^{j,k}\partial _{x}^{j_{1}}\left[ K^{\left(
k_{1}\right) }\left( b\left( \left\vert x\right\vert \right) x\theta \right)
\left( b\left( \left\vert x\right\vert \right) x\right) ^{k_{1}}\right]
\partial _{x}^{j_{2}}\left[ \left( -i\varphi \left( x\right) \right)
^{k_{2}}e^{-i\varphi \left( x\right) \theta }\right] ,
\end{gather*}%
where $C_{j_{1},j_{2},k_{1},k_{2}}^{j,k}=\frac{j!k!}{j_{1}!j_{2}!k_{1}!k_{2}!%
}.$ Then, by means of $\left( \ref{Es}\right) $ and $\left( \ref{C}\right) ,$
we have for constants $C_{j,k}^{\prime }$%
\begin{gather}
\left\vert \partial _{x}^{j}\partial _{\theta }^{k}\left[ e^{-i\varphi
\left( x\right) \theta }K\left( b\left( \left\vert x\right\vert \right)
x\theta \right) \right] \right\vert \leq C_{j,k}^{\prime }K_{j+k}\left(
b\left( \left\vert x\right\vert \right) x\theta \right)  \notag \\
\sum\limits_{\substack{ j_{1}+j_{2}=j  \\ k_{1}+k_{2}=k}}\max_{\substack{ %
s_{1}+s_{2}+s_{3}=j_{1}  \\ s_{3}\leq k_{1}}}\left\{ \left\vert \theta
\right\vert ^{s_{1}}\left\vert x\right\vert ^{-s_{2}}\left( b\left(
\left\vert x\right\vert \right) x\right) ^{k_{1}-s_{3}}\right\} \max 
_{\substack{ l_{1}+l_{2}=j_{2}  \\ l_{2}\leq k_{2}}}\left\{ \left\vert
\theta \right\vert ^{l_{1}}\left\vert x\right\vert ^{k_{2}-l_{2}}\right\} ,
\label{2.13}
\end{gather}%
where $K_{p}\left( t\right) ,$ $p\in \mathbb{N}$ are defined by%
\begin{equation*}
K_{0}\left( t\right) =\left\vert K\left( t\right) \right\vert ,\text{ }%
K_{p}\left( t\right) =\max_{1\leq p^{\prime }\leq p}\left\vert K^{\left(
p^{\prime }\right) }\left( t\right) \right\vert ,\text{ }p\in \mathbb{N}%
^{\ast }.
\end{equation*}%
Writing $\left\vert b\left( \left\vert x\right\vert \right) x\right\vert
=\left\vert b\left( \left\vert x\right\vert \right) x\theta \right\vert
\left\vert \theta \right\vert ^{-1}$ and $\left\vert x\right\vert
^{-1}=\left\vert b\left( \left\vert x\right\vert \right) x\theta \right\vert
^{-1}b\left( \left\vert x\right\vert \right) \left\vert \theta \right\vert ,$
then there exists a constant $C,$%
\begin{equation}
\left\{ 
\begin{array}{l}
\left\vert b\left( \left\vert x\right\vert \right) x\right\vert \leq
C\lambda \left( b\left( \left\vert x\right\vert \right) x\theta \right)
\lambda ^{-1}\left( \theta \right) , \\ 
\left\vert x\right\vert ^{-1}\leq C\left\vert b\left( \left\vert
x\right\vert \right) x\theta \right\vert ^{-1}\lambda \left( \theta \right)%
\end{array}%
\right. \text{ on }\left[ -1,1\right] \times \mathbb{R}_{\theta }.
\label{2.14}
\end{equation}%
We have $b^{-1}\left( \left\vert x\right\vert \right) \leq b^{-1}\left(
\lambda ^{-1}\left( \theta \right) \right) $ when $\left\vert x\right\vert
\geq \lambda ^{-1}\left( \theta \right) $ (because $b$ is increasing). Then,%
\begin{gather*}
\left\vert x\right\vert =\left\vert b\left( \left\vert x\right\vert \right)
x\theta \right\vert \left( b\left( \left\vert x\right\vert \right)
\left\vert \theta \right\vert \right) ^{-1}\leq \left\vert b\left(
\left\vert x\right\vert \right) x\theta \right\vert \left\vert \theta
\right\vert ^{-1}b^{-1}\left( \lambda ^{-1}\left( \theta \right) \right) \\
\leq C\lambda \left( b\left( \left\vert x\right\vert \right) x\theta \right)
\lambda ^{-1}\left( \theta \right) b^{-1}\left( \lambda ^{-1}\left( \theta
\right) \right) .
\end{gather*}%
Bearing in mind the other case when $\left\vert x\right\vert \leq $ $\lambda
^{-1}\left( \theta \right) ,$ we obtain for a constant $\widetilde{C}$ 
\begin{equation}
\left\vert x\right\vert \leq \widetilde{C}\lambda \left( b\left( \left\vert
x\right\vert \right) x\theta \right) \lambda ^{-1}\left( \theta \right)
b^{-1}\left( \lambda ^{-1}\left( \theta \right) \right) .  \label{2.15}
\end{equation}%
Finally from $\left( \ref{2.11}\right) $ to $\left( \ref{2.15}\right) $, we
obtain $\left( \ref{Es2}\right) .$
\end{proof}

\begin{lemma}
\label{Lemma2}\textit{For any continuous function }$b_{0}\left( t\right) $ 
\textit{on} $\left[ 1,+\infty \right[ $ \textit{such that}%
\begin{equation}
b_{0}\left( t\right) >0,\text{ \ \ }\lim_{t\rightarrow +\infty }b_{0}\left(
t\right) =+\infty ,  \label{2.16}
\end{equation}%
\textit{then there exists a continuous function }$b\left( t\right) $ \textit{%
on }$\left[ 0,1\right] $\textit{\ witch satisfies conditions }$\left( \ref{C}%
\right) $ \textit{such that we have on }$\left[ -1,1\right] ^{n}\times 
\mathbb{R}_{\theta }^{n}$%
\begin{equation}
\left\vert \partial _{x}^{\alpha }\partial _{\theta }^{\beta }q(x,\theta
)\right\vert \leq C_{\alpha ,\beta }\lambda ^{\left\vert \alpha \right\vert
-\left\vert \beta \right\vert }\left( \theta \right) \left\{ b_{0}\left(
\lambda \left( \theta \right) \right) \right\} ^{\left\vert \beta
\right\vert },\text{ }\alpha ,\beta \in \mathbb{N}^{n}\text{.}  \label{2.17}
\end{equation}
\end{lemma}

\begin{proof}
Setting%
\begin{equation*}
\left\{ 
\begin{array}{l}
f_{0}\left( t\right) =\left\{ b_{0}\left( t^{-1}\right) \right\} ^{-1}\text{
on }\left] 0,1\right] \\ 
f_{0}\left( 0\right) =0,%
\end{array}%
\right.
\end{equation*}%
then $f_{0}$ is a continuous function on $\left[ 0,1\right] $ which verifies
condition $\left( \ref{2.1}\right) .$ Then, by lemma \ref{LemmaKumano},
there exists a continuous function $b\left( t\right) $ which satisfies $%
\left( \ref{C}\right) $. Noting that%
\begin{equation*}
\left\{ b\left( \lambda ^{-1}\left( \theta \right) \right) \right\}
^{-1}\leq \left\{ f_{0}\left( \lambda ^{-1}\left( \theta \right) \right)
\right\} ^{-1}=b_{0}\left( \lambda \left( \theta \right) \right)
\end{equation*}%
this gives $\left( \ref{2.17}\right) .$
\end{proof}

\begin{lemma}
\label{Lemma3}Let $\left\{ b_{l}\left( t\right) \right\} _{l\in \mathbb{N}%
^{\ast }}$ be a sequence of continuous functions on $\left[ 1,+\infty \right[
$ \ which satisfy $\left( \ref{2.16}\right) .$ Then, there exists a
continuous function $b_{0}\left( t\right) $ verifiying $\left( \ref{2.16}%
\right) ,$ such that, for any $l_{0},$%
\begin{equation*}
b_{l}\left( t\right) \geq b_{0}\left( t\right) \text{ on }\left[
t_{l_{0}},+\infty \right[ ,\text{ }l=1,...,l_{0}\text{ }
\end{equation*}
\end{lemma}

Finally, our but is to give an unbounded Fourier integral operator of the
form $\left( \ref{1.3}\right) $ with symbol in $\bigcap\limits_{0<\rho
<1}S_{\rho ,1}^{0}\left( \mathbb{R}_{x}^{n}\times \mathbb{R}_{\theta
}^{n}\right) .$

\begin{theorem}
There exist a Fourier integral operator $F$ of the form $\left( \ref{1.3}%
\right) ,$ with symbol $a\in \bigcap\limits_{0<\rho <1}S_{\rho ,1}^{0}\left( 
\mathbb{R}_{x}^{n}\times \mathbb{R}_{\theta }^{n}\right) ,$ which cannot be
extended to be a bounded operator on $L^{2}\left( \mathbb{R}^{n}\right) .$
\end{theorem}

\begin{proof}
Let $\phi \left( s\right) $ be a-$C_{0}^{\infty }\left( \mathbb{R}\right) $
function such that%
\begin{equation*}
\left\{ 
\begin{array}{l}
\phi \left( s\right) =1~\ \text{on }\left[ -\beta ,\beta \right] \text{ \ }%
(\beta <1) \\ 
\text{supp}\phi \subset \left[ -1,1\right] .%
\end{array}%
\right.
\end{equation*}%
Define a $C^{\infty }$-symbol $a\left( x,\theta \right) $ by%
\begin{equation*}
a\left( x,\theta \right) =e^{-i\varphi \left( x\right) \psi \left( \theta
\right) }\prod\limits_{j=1}^{n}\phi \left( x_{j}\right) K\left( b\left(
\left\vert x\right\vert \right) \left\vert x\right\vert \theta _{j}\right) 
\text{ in }\mathbb{R}_{x}^{n}\times \mathbb{R}_{\theta }^{n}
\end{equation*}%
where $K\left( t\right) $ and $b\left( t\right) $ are the functions of lemma %
\ref{Lemma1}. Let $b_{l}\left( t\right) =\overset{l}{\overbrace{\log ...\log 
}}\left( C_{l}+t\right) $ defined on $\left[ 1,+\infty \right[ $ and $C_{l}$
some large constant, then by lemmas \ref{Lemma2} and \ref{Lemma3} we have $%
,\forall (\alpha ,\gamma )\in \mathbb{N}^{n}\times \mathbb{N}^{n},$ $\forall
l\in \mathbb{N}^{\ast }$ 
\begin{equation}
\left\vert \partial _{x}^{\alpha }\partial _{\theta }^{\gamma }a\left(
x,\theta \right) \right\vert \leq C_{\alpha ,\gamma ,l}\lambda ^{\left\vert
\alpha \right\vert -\left\vert \gamma \right\vert }\left( \theta \right)
\left\{ b_{l}\left( \lambda \left( \theta \right) \right) \right\}
^{\left\vert \gamma \right\vert },  \label{2.18}
\end{equation}%
$C_{\alpha ,\gamma ,l}$ are constants, so that $a\left( x,\theta \right) \in
\bigcap\limits_{0<\rho <1}S_{\rho ,1}^{0}\left( \mathbb{R}_{x}^{n}\times 
\mathbb{R}_{\theta }^{n}\right) .$ Furthermore the corresponding Fourier
integral $F$ is%
\begin{eqnarray}
\left( Fu\right) \left( x\right) &=&\int_{\mathbb{R}_{\theta
}^{n}}e^{i\varphi \left( x\right) \psi \left( \theta \right) }a\left(
x,\theta \right) \mathcal{F}u\left( \theta \right) d\theta  \notag \\
&=&\prod\limits_{j=1}^{n}\phi \left( x_{j}\right) \int_{\mathbb{R}_{\theta
}^{n}}\prod\limits_{j=1}^{n}K\left( b\left( \left\vert x\right\vert \right)
\left\vert x\right\vert \theta _{j}\right) \mathcal{F}u\left( \theta \right)
d\theta ,\text{ \ }u\in \mathcal{S}(\mathbb{R}^{n}).  \label{2.19}
\end{eqnarray}%
We consider $\left( Fu\right) \left( x\right) $ in $\left] 0,\beta \right]
^{n}.$ Then, using an adequate change of variable in the integral $\left( %
\ref{2.19}\right) $, we have%
\begin{equation*}
\left( Fu\right) \left( x\right) =\int_{\mathbb{R}_{z}^{n}}u\left( b\left(
\left\vert x\right\vert \right) \left\vert x\right\vert z\right)
\prod\limits_{j=1}^{n}\mathcal{F}K\left( z_{j}\right) dz
\end{equation*}%
which has the form of $A$ in theorem \ref{Theorem1}. In addition the
function $g\left( x\right) =b\left( \left\vert x\right\vert \right)
\left\vert x\right\vert $ satisfies\textit{\ }$\left( \ref{2.6}\right) .$
Consequently the operator $F$ cannot be extended as a bounded operator on $%
L^{2}\left( \mathbb{R}^{n}\right) .$
\end{proof}

\end{document}